\documentclass[12pt,leqno,amscd,amsfonts,amssymb]{amsart}
\usepackage{amsmath,amssymb,amscd}
\usepackage{mathrsfs}
\usepackage{hyperref}
\usepackage{bookmark}
\usepackage{bm}
\usepackage{amsthm}
\usepackage{bookmark}
\usepackage{cite}

\def\la{\lambda}
\def\a{\alpha}
\def\b{\beta}
\def\r{\gamma}

\def\Z{\mathbb{Z}}
\def\N{\mathbb{N}}

\def\C{\mathbb{C}}

\numberwithin{equation}{section}
\newtheorem{theo}{Theorem}[section]
\newtheorem{defi}[theo]{Definition}
\newtheorem{coro}[theo]{Corollary}
\newtheorem{lemm}[theo]{Lemma}
\newtheorem{prop}[theo]{Proposition}

\allowdisplaybreaks

\begin{document}
	
	\title[A new class of irreducible  modules over the BMS-Kac-Moody algebra]{A new class of irreducible  modules over the BMS-Kac-Moody algebra}
	
	\author{Qiu-Fan Chen, Yi-Jun Bai}
	
	\address{Department of Mathematics, Shanghai Maritime University,
		Shanghai, 201306, China.}
	\email{chenqf@shmtu.edu.cn}
	
	\address{Department of Mathematics, Shanghai Maritime University,
		Shanghai, 201306, China.}
	\email{202431010011@stu.shmtu.edu.cn}
	
	
	
	\thanks{This work is supported by Science and Technology Commission of Shanghai Municipality (Grant No. 25ZR1402183) and National Natural Science Foundation of China (Grant Nos. 12271345 and 12361006).}
	
\begin{abstract}
In this paper, we construct a class of non-weight modules over the BMS-Kac-Moody algebra by taking tensor products of  finitely many  irreducible modules $\Phi(\lambda,\a,\b,\r,h(t))$ with irreducible restricted modules. We obtain the necessary and sufficient conditions for these tensor product modules to be irreducible, and determine the corresponding conditions for two such modules to be isomorphic. Moreover, we compare these modules with other known non-weight modules, showing that these irreducible modules are new.
\end{abstract}
	
\maketitle
	
\tableofcontents
	
\section{Introduction}
Throughout the paper, we denote by $\C ,\,\C^*,\,\Z,\,\Z_+,\,\N$ the sets of complex numbers, nonzero complex numbers, integers, nonnegative integers and positive integers, respectively. All algebras (modules, vector spaces) are assumed to be  over $\C$. For a Lie algebra $\mathfrak{g}$, we use $U(\mathfrak{g})$ to denote the universal enveloping algebra of $\mathfrak{g}$. Let $\C[t]$ be the polynomial algebra in variable $t$ and $\C[s,t]$ the polynomial algebra in variables $s$ and $t$, respectively.
	
Infinite-dimensional symmetries play a prominent role in  various areas of physics. In particular, the asymptotic symmetry group of three-dimensional Einstein gravity has been shown to be generated by the BMS (Bondi-Metzner-Sachs) algebra, which is an infinite-dimensional Lie algebra with a basis $\{L_m, W_m, C_L, C_W\mid m\in\Z\}$ and the nontrivial defining relations
\begin{equation*}
\aligned
&[L_m, L_n]=(n-m)L_{m+n}+\frac{C_L}{12}(m^3-m)\delta_{m+n,0},\\
&[L_m,W_n]=(n-m)W_{m+n}+\frac{C_W}{12}(m^3-m)\delta_{m+n,0},\\
&[W_m,W_n]=0, \quad \forall m, n\in\Z.
\endaligned
\end{equation*}
These $L_m$ and $W_m$ correspond respectively to super-rotations and super-translations at the quantum level. The BMS algebra is  precisely  the $W(2,2)$ algebra, which was first introduced in \cite{ZD} for the classification of moonshine type vertex operator algebras generated by two weight 2 vectors.  We note that the BMS algebra is crucial in establishing holography theory in asymptotic flat spacetimes, see\cite{B,BF,H}. Owing to an additional anisotropic scaling symmetry, the BMS symmetry has been extended to a symmetry generated by a new type of BMS-Kac-Moody algebra with a nonvanishing $u(1)$ Kac-Moody current\cite{YC}. Motivated by \cite{YC}, an infinite-dimensional BMS-Kac-Moody algebra $\mathfrak L$ with two $u(1)$  Kac-Moody currents without central extensions was introduced in \cite{LS} and many irreducible restricted modules over $\mathfrak L$ were constructed therein. Explicitly, $\mathfrak L$ is a Lie algebra with a basis $\{L_m, W_m, I_m, J_m,  \mid m\in\Z\}$ and the nontrivial Lie brackets defined by
\begin{equation*}
\aligned
&[L_m,L_n]=(n-m)L_{m+n},\quad [L_m,W_n]=(n-m)W_{m+n},\\
&[L_m,I_n]=nI_{m+n},\quad [L_m,J_n]=nJ_{m+n},\\
&[W_m,I_n]=-nJ_{m+n},\quad \forall\, m,n\in\Z.
\endaligned
\end{equation*}
In addition to the BMS algebra,  $\mathfrak L$  contains several other well-known subalgebras, such as the Lie subalgebra spanned by $\{L_m\mid m\in\Z\}$ is the Witt algebra and the Lie subalgebra spanned by $\{L_m, I_m\mid m\in\Z\}$ and   $\{L_m, J_m\mid m\in\Z\}$ are two copies of the (centerless) Heisenberg-Virasoro algebra. Moreover, the Lie subalgebra spanned by $\{I_m, J_m\mid m\in\Z\}$ is a commutative ideal of $\mathfrak L$. Notice that the Lie subalgebra $\C L_0\oplus\C W_0$  is  the Cartan subalgebra  (modulo center) of $\mathfrak{L}$. The construction of 
$\mathfrak L$ suggests that such an infinite-dimensional algebra is of significant interest for both structural and representation theory. Moreover, $\mathfrak L$ admits a connection to the two-dimensional supersymmetric Galilean conformal algebra \cite{MR}, as it arises as the even part of that algebra.

In recent years, various constructions of irreducible non-Harish-Chandra modules and non-weight modules have been introduced.  In particular, J. Nilsson \cite{N} constructed  a class of  $\mathfrak{sl_{n+1}}$-modules that are free of rank one when restricted to the Cartan subalgebra, which were also introduced by a different method\cite{TZ}.  Since then, such non-weight modules, commonly referred to as $U(\mathfrak{h})$-free modules, have been extensively studied. More specifically, the authors of  \cite{CG} classified the $U(\C L_0\oplus\C W_0)$-free modules of rank one for  $\mathfrak L$, denoted by  $\Phi(\lambda,\a,\b,\r,h(t))$. Furthermore, the irreducibility and isomorphism classes of these modules were determined therein. Restricted modules for a $\Z$-graded Lie algebra are those in which every vector is annihilated by the sufficiently large positive part of the algebra. Whittaker modules and highest weight modules are restricted modules, and, in some sense,  restricted modules can be viewed as a generalization of both  Whittaker modules and highest weight modules. A fundamental step in the study of restricted modules is the classification of all irreducible restricted modules for a given Lie algebra. This was first accomplished for the Virasoro algebra  by V. Mazorchuk  and K. Zhao in \cite{MZ}, and subsequently, analogous results have been obtained for many other $\Z$-graded Lie algebras. For the algebra $\mathfrak L$,  partial results on irreducible restricted modules have been obtained in \cite{LS}.

It is well known that an important way to construct new modules over an algebra is to consider the linear tensor product of known modules over the algebra, see \cite{CGZ,CHSY,CY,GLW,TZ2,TZ3,Z}. The aim of this paper is to construct new irreducible non-weight  $\mathfrak L$-modules by taking tensor products of finitely many irreducible  modules $\Phi(\lambda,\a,\b,\r,h(t))$ with irreducible restricted modules.
	
This paper is organized as follows. Section 2 recalls the definitions of the various modules considered herein, along with some basic known results. In Section 3, we construct a family of $\mathfrak L$-modules by taking tensor products of a finite number of $\Phi(\lambda,\a,\b,\r,h(t))$ with an arbitrary irreducible restricted module $V$. We then establish necessary and sufficient conditions for the irreducibility of these tensor products and investigate their submodule structures in the reducible cases. Section 4 is devoted to determining the isomorphism criteria for any two such irreducible tensor modules. In Section 5, we compare the irreducible modules constructed in Section 3 with other known non-weight irreducible modules, demonstrating that they indeed form a new class of irreducible $\mathfrak L$-modules.

\section{Preliminaries}\label{pre}
Fix any $\lambda\in\C^*,\a\in\C$ and $h(t)\in\C[t]$. For any $n\in\Z$, we define
\begin{align*}
h_{(n)}(t)=n h(t)-n(n-1)\a\frac{h(t)-h(\a)}{t-\a}.
\end{align*}

For $\lambda\in\C^*,\a,\b,\r\in\C$ and $h(t)\in\C[t]$, denote by $\Phi(\lambda,\a,\b,\r,h(t))=\C [s,t]$. In \cite{CG}, the $\mathfrak L$-module structure on $\Phi(\lambda,\a,\b,\r,h(t))$ is given by
\begin{equation*}
\aligned
&L_n(f(s,t))=\lambda^n(s+h_{(n)}(t))f(s-n,t)-n\lambda^n(t-n\a)\frac{\partial}{\partial t}\big(f(s-n,t)\big),\\
&W_n(f(s,t))=\lambda^n(t-n\a)f(s-n,t),\\
&I_n(f(s,t))=\la^n\big(\r-n\b \frac{h(t)-h(\a)}{t-\a}\big)f(s-n,t)+n\la^n\b \frac{\partial}{\partial t}\big(f(s-n,t)\big),\\
&J_n(f(s,t))=\la^n\b f(s-n,t), \quad {\rm where}\,\, f(s,t)\in\C[s,t].
\endaligned
\end{equation*}
For convenience, we denote $H(t)=\frac{h(t)-h(\a)}{t-\a}$.   Then, the action of $L_n$ and $I_n$ on $\Phi(\lambda,\a,\b,\r,h(t))$ can be written respectively as
\begin{align*}
L_n(f(t)s^i)&=\lambda^n(s-n)^i\big(sf(t)+n((tH(t)+h(\a))f(t)-tf^{\prime}(t))-n^2\a(H(t)f(t)-f^{\prime}(t))\big),\\
I_n(f(t)s^i)&=\lambda^n(s-n)^i\big(\r f(t)+n\b(f^{\prime}(t)-f(t)H(t))\big).
\end{align*}

For later use, we need the following known results of $\Phi(\lambda,\a,\b,\r,h(t))$.
\begin{prop}\label{pop1}(cf. \cite{CG})
Keep notations as above, then the following statements hold.
\begin{itemize}
\item[(1)] Any free $U(\C L_0\oplus\C W_0)$-module of rank one over  $\mathfrak L$  is isomorphic to $\Phi(\lambda,\a,\b,\r,h(t))$ for $\lambda\in\C^*,\a,\b,\r\in\C, h(t)\in\C[t]$.
\item[(2)] $\Phi(\lambda,\a,\b,\r,h(t))$  is irreducible  if and only if $\a\neq0$ or $\b\neq0$.
\item[(3)] $\Phi(\lambda,\a,\b,\r,h(t))\cong \Phi(\lambda^{\prime},\a^{\prime},\b^{\prime},\r^\prime,g(t))\Longleftrightarrow (\lambda,\a,\b,\r,h(t))=(\lambda^{\prime},\a^{\prime},\b^{\prime},\r^\prime,g(t)).$
\end{itemize}
\end{prop}
The Lie algebra $\mathfrak L$ is a $\Z$-graded algebra with a decomposition
$$\mathfrak L=\oplus_{n\in\Z}\mathfrak L_{n},\quad [\mathfrak L_{n}, \mathfrak L_{m}]\subset\mathfrak L_{n+m},\,\,m,n\in\Z,$$
where  $\mathfrak L_{n}=\C L_n\oplus \C W_n\oplus \C I_n\oplus \C J_n$.
\begin{defi}\label{defi2.2}
An $\mathfrak L$-module $V$ is called a restricted module if for any $v \in V$, there exists $m \in \mathbb{Z}_{+}$ such that $\mathfrak L_n \cdot v = 0$ for all $n \geq m $. 
\end{defi}
    
In the rest of this paper, we shall always assume that $\r \in\C$, $(\lambda, \a)\in(\C^*)^2$ or $(\lambda, \b)\in(\C^*)^2$, $h(t) \in \mathbb{C}[t]$ and $V$ is an irreducible restricted $\mathfrak L$-module.
    
\section{Irreducibility of tensor product modules}
%	In this section, we investigate the structures of the tensor products of the irreducible  $\mathfrak L$-modules defined in the previous section, that is, the modules $\Omega(\lambda, \alpha, h)$ and simple modules as in Definition \ref{defi2.2}. In particular, we determine their irreducibility.
Let $s_1, s_2, \ldots, s_m, t_1, t_2, \ldots, t_m$ be commuting variables, where $m\in\N$. Fix any complex numbers $\r_i \in\C$, $(\lambda_i, \a_i)\in(\C^*)^2$ or  $(\lambda_i, \b_i)\in(\C^*)^2$, and polynomials $h_i(t_i) \in \C[t_i]$, $1\leq i \leq m$, we have the corresponding $\mathfrak L$-modules $\Phi(\lambda_i, \alpha_i, \b_i, \r_i, h_i(t_i))$, $1\leq i \leq m$ and $\Phi(\lambda_i, \alpha_i, \b_i, \r_i, h_i(t_i))=\C[s_i, t_i]$ for $1\leq i \leq m$ as vector spaces. 
Taking an arbitrary  irreducible  restricted $\mathfrak L$-module $V$, we can form the tensor product
\begin{equation}\label{a11} \bm{T}(\bm{\lambda}, \bm{\alpha}, \bm{\b}, \bm{\r}, \bm{h(t)}, V)=(\otimes_{i=1}^{m}\Phi(\lambda_i, \alpha_i, \b_i, \r_i, h_i(t_i))\otimes V,\end{equation}
where $\bm{\lambda}=(\lambda_1, \ldots, \lambda_m), \bm{\alpha}=(\alpha_1, \ldots, \alpha_m),\bm{\b}=(\b_1, \ldots, \b_m),\bm{\r}=(\r_1, \ldots, \r_m)$ and $\bm{h(t)}=(h_1(t_1), \ldots, h_m(t_m))$. Denote simply $\bm{T}:=\bm{T}(\bm{\lambda}, \bm{\alpha}, \bm{\b}, \bm{\r}, \bm{h(t)}, V)$. Take any nonzero element
\begin{equation*}f(s_1, s_2, \ldots, s_m, t_1, t_2, \ldots, t_m)=\sum_{(\mathbf{p}, \mathbf{q})\in E}s_1^{p_1}t_1^{q_1}\otimes\cdots\otimes s_m^{p_m}t_m^{q_m}\otimes v_{(\mathbf{p}, \mathbf{q})}\in\bm{T},\end{equation*}
where $(\mathbf{p}, \mathbf{q})=(p_1, p_2, \ldots, p_m, q_1, q_2, \ldots, q_m), v_{(\mathbf{p}, \mathbf{q})}\in V\setminus\{0\}$ and $E$ is a finite subset of $(\mathbb{Z}_{+})^{2m}$.  For convenience, we may write $\mathbf{s}^{\mathbf{p}}\mathbf{t}^{\mathbf{q}}=s_1^{p_1}t_1^{q_1}\otimes\cdots \otimes s_m^{p_m}t_m^{q_m}$
and the element $f(\mathbf{s}, \mathbf{t})\in \bm{T}$ can be rewritten as
\begin{equation}\label{a12}f(\mathbf{s}, \mathbf{t})=\sum_{(\mathbf{p}, \mathbf{q})\in E}\mathbf{s}^{\mathbf{p}}\mathbf{t}^{\mathbf{q}}\otimes v_{(\mathbf{p}, \mathbf{q})}.\end{equation}
Denote $P_i= {\rm max\,}\{p_i\mid (\mathbf{p}, \mathbf{q})\in E\}$ for all $1\leq i \leq m$. We simply write $(\mathbf{s}^{\mathbf{p}}\mathbf{t}^{\mathbf{q}})(\mathbf{s}^{\mathbf{p\prime}}\mathbf{t}^{\mathbf{q\prime}})=\mathbf{s}^{\mathbf{p}+\mathbf{p\prime}}\mathbf{t}^{\mathbf{q}+\mathbf{q\prime}}$ for short in what follows.

For any positive integer $l$,  we can define a total order $``\succ"$ on $\Z^l$ by
$$(a_1, \ldots, a_l) \succ (b_1, \ldots, b_l) \Longleftrightarrow  \mathrm{ there\ exists} \ k\in\N \ \mathrm{such \ that}  \ a_k>b_k \ \mathrm{and}\ a_{i}=b_{i},\ \forall\, 1\leq i<k.$$
Then we define the degree ${\rm deg\,}(f(\mathbf{s}, \mathbf{t}))$ of $f(\mathbf{s}, \mathbf{t})$ as the maximal $(\mathbf{p}, \mathbf{q})\in E$ with $v_{(\mathbf{p}, \mathbf{q})}\neq0$. Note that ${\rm deg\,}(1\otimes \cdots \otimes 1\otimes v)=\mathbf{0}=(0, 0,\ldots, 0)$ for $v\in V\setminus \{0\}$.	
	
We need the following two crucial results which will be used repeatedly throughout the paper.
\begin{lemm}\label{prop1}(cf. \cite{TZ2})
Let $\lambda_1, \lambda_2, \ldots, \lambda_m\in\C^*$,  $k_1, k_2, \ldots, k_m\in\N$ with $k_1+k_2+\cdots+k_m=k$. Define a sequence of functions on $\Z$ as follows: $f_1(n)=\lambda_1^n, f_2(n)=n\lambda_1^n, \ldots, f_{k_1}(n)=n^{k_1-1}\lambda_1^n, f_{k_1+1}(n)=\lambda_2^n, \ldots, f_{k_1+k_2}(n)=n^{k_2-1}\lambda_2^n, \ldots, f_{k}(n)=n^{k_m-1}\lambda_m^n$. Let $\mathfrak{M}=(y_{pq})$ be the $k\times k$ matrix with $y_{pq}=f_q(p-1)$, $q=1, 2, \ldots, k, p=r+1, r+2, \ldots, r+k$ where $r\in\Z_+$. Then
\begin{equation*}det(\mathfrak{M})=\prod_{j=1}^m(k_j-1)!!\lambda_j^{k_j(k_j+2r-1)/2}\prod_{1\leq i<j\leq m}(\lambda_j-\lambda_i)^{k_ik_j},\end{equation*}
where $k_j!!=k_j!\times (k_j-1)!\times\cdots \times 2!\times 1!$ with $0!!=1$.
	\end{lemm}
	\begin{prop}\label{prop2}
Assume that $\lambda_1, \ldots, \lambda_m$ are distinct.  Let $M$ be a subspace of $\bm{T}$ that is stable under the actions of $L_n, W_n, I_n$ and $J_n$ for any $n$ sufficiently large. Then for any nonzero $f(\mathbf{s}, \mathbf{t})=\sum_{(\mathbf{p}, \mathbf{q})\in E}\mathbf{s}^{\mathbf{p}}\mathbf{t}^{\mathbf{q}}\otimes v_{(\mathbf{p}, \mathbf{q})}\in M$,  we have
\begin{align}\label{111}&\sum_{(\mathbf{p}, \mathbf{q})\in E}\mathbf{s}^{\mathbf{p}+\omega_i}\mathbf{t}^{\mathbf{q}}\otimes v_{(\mathbf{p}, \mathbf{q})}\in M,\\
\label{222}&\sum_{(\mathbf{p}, \mathbf{q})\in E}\mathbf{s}^{\mathbf{p}}\mathbf{t}^{\mathbf{q}+\omega_i}\otimes v_{(\mathbf{p}, \mathbf{q})}\in M,\\
\label{333}&\sum_{\substack{(\mathbf{p}, \mathbf{q})\in E \\ p_i=P_i}}\mathbf{s}^{\mathbf{p}-P_i\omega_i}\big(H_i(t_i)\mathbf{t}^{\mathbf{q}}-q_i\mathbf{t}^{\mathbf{q}-\omega_i}\big)\otimes v_{(\mathbf{p}, \mathbf{q})}\in M,\\
\label{444}&\sum_{\substack{(\mathbf{p}, \mathbf{q})\in E \\ p_i=P_i}}\mathbf{s}^{\mathbf{p}-P_i\omega_i}\mathbf{t}^{\mathbf{q}}\otimes v_{(\mathbf{p}, \mathbf{q})}\in M,
\end{align}
where $\omega_i=(\delta_{i, 1}, \delta_{i, 2},\ldots, \delta_{i, m})$ and $i=1, \ldots, m$.
\end{prop}
\begin{proof}
The case $\a\in\C^*$ was proved in \cite[Lemma 3.2]{CX} by considering $L_n(f(\mathbf{s}, \mathbf{t}))$ and $W_n(f(\mathbf{s}, \mathbf{t}))$ for sufficiently large integer $n$.
Note that the proofs \eqref{111} and \eqref{222} are independent of $\a$.  Therefore, it remains to prove \eqref{333} and \eqref{444} for the case $\b\in\C^*$. For sufficiently large integer $n$, we have
\begin{align}\label{bbn}
I_n(f(\mathbf{s}, \mathbf{t}))&=\sum_{(\mathbf{p}, \mathbf{q})\in E}\sum_{i=1}^ms_1^{p_1}t_1^{q_1}\otimes\cdots \otimes I_n(s_i^{p_{i}}t_i^{q_i})\otimes\cdots\otimes s_m^{p_m}t_m^{q_m}\otimes v_{(\mathbf{p}, \mathbf{q})}\nonumber\\
&=\sum_{(\mathbf{p}, \mathbf{q})\in E}\sum_{i=1}^{m} s_1^{p_1} t_1^{q_1} \otimes \cdots \otimes \lambda_i^n(s_i-n)^{p_i}\big(\r_i t_i^{q_i}+n\b_i(q_it_i^{q_i-1}-H_i(t_i) t_i^{q_i})\big) \nonumber\\
&\quad\otimes \cdots \otimes s_m^{p_m} t_m^{q_m} \otimes v_{(\mathbf{p}, \mathbf{q})} \nonumber\\
&=\sum_{(\mathbf{p}, \mathbf{q}) \in E} \sum_{i=1}^{m} \sum_{j=0}^{p_i+1} (-1)^j n^j\lambda_i^n  s_1^{p_1} t_1^{q_1} \otimes \cdots \otimes \big(\r_i\binom{p_i}{j} s_i^{p_i-j} t_i^{q_i} \nonumber\\
&\quad-\b_i\binom{p_i}{j-1}s_i^{p_i-j+1}(q_it_i^{q_i-1}-H_i(t_i) t_i^{q_i})\big)\otimes \cdots\otimes s_m^{p_m} t_m^{q_m} \otimes v_{(\mathbf{p}, \mathbf{q})} 
\end{align}
and
\begin{align}\label{bbn2}
J_n(f(\mathbf{s}, \mathbf{t}))&=\sum_{(\mathbf{p}, \mathbf{q})\in E}\sum_{i=1}^ms_1^{p_1}t_1^{q_1}\otimes\cdots \otimes J_n(s_i^{p_{i}}t_i^{q_i})\otimes\cdots\otimes s_m^{p_m}t_m^{q_m}\otimes v_{(\mathbf{p}, \mathbf{q})}\nonumber\\
&=\sum_{(\mathbf{p}, \mathbf{q})\in E}\sum_{i=1}^ms_1^{p_1}t_1^{q_1}\otimes\cdots \otimes \lambda_i^n\b_i(s_i-n)^{p_i} t_i^{q_i}\otimes\cdots\otimes s_m^{p_m}t_m^{q_m}\otimes v_{(\mathbf{p}, \mathbf{q})}\nonumber\\
&=\sum_{(\mathbf{p}, \mathbf{q})\in E}\sum_{i=1}^{m}\sum_{k=0}^{p_i}(-1)^k n^k\lambda_i^n\b_i s_1^{p_1} t_1^{q_1}\otimes\cdots\otimes\binom{p_i}{k} s_i^{p_{i}-k} t_i^{q_i}\otimes\cdots\otimes s_m^{p_m}t_m^{q_m}\otimes v_{(\mathbf{p}, \mathbf{q})},
\end{align}
where we make the convention that $\binom{0}{0}=1$ and $\binom{x}{y}=0$ whenever $y>x$ or $y<0$.  Thanks to Lemma \ref{prop1}, we know that the coefficients of $n^{P_i+1}\lambda_i^n $ and $n^{P_i}\lambda_i^n$ in \eqref{bbn} and \eqref{bbn2}  lie in  $M$, giving  \eqref{333} and \eqref{444}, respectively. This completes the proof.
%\begin{align*}\sum_{(\mathbf{p}, \mathbf{q})\in E}s_1^{p_1} t_1^{q_1} \otimes \cdots \otimes \big(\r_i\binom{p_i}{j} s_i^{p_i-j} t_i^{q_i}
%-\b_i\binom{p_i}{j-1}s_i^{p_i-j+1}t_i^{q_i-1}(q_i-h_i(t_i)+h_i(0))\big)\\
%&\quad\otimes \cdots \otimes s_m^{p_m} t_m^{q_m} \otimes v_{(\mathbf{p}, \mathbf{q})}\in M,\\
%b_{i,k}&:=\sum_{(\mathbf{p}, \mathbf{q})\in E}s_1^{p_1} t_1^{q_1}\otimes\cdots\otimes\binom{p_i}{k} s_i^{p_{i}-k} t_i^{q_i}\otimes\cdots\otimes s_m^{p_m}t_m^{q_m}\otimes v_{(\mathbf{p}, \mathbf{q})}\in M.
%\end{align*}
%For any $1\leq i \leq m$, taking $j=0,P_i+2$ and $k=0,P_i+1$ in the above two elements, respectively, we get 
\end{proof}
\begin{prop}\label{prop3}
Let $\lambda_1, \ldots, \lambda_m$ be pairwise distinct. Then $1\otimes\cdots \otimes 1\otimes v$ generates the  $\mathfrak L$-module $\bm{T}$ for  any $0\neq v\in V$.
\end{prop}
\begin{proof}
%Since the proof is analogous to that of \cite[Lemma 3.3]{CX}, we omit the details.
Fix any nonzero $v\in V$ and let $M$ denote the submodule of $\bm{T}$ generated by $1\otimes\cdots \otimes 1\otimes v$. Using \eqref{222} and and inducting on  $\mathbf{q}$, we have $\mathbf{t}^{\mathbf{q}}\otimes v\in M$ for all $\mathbf{q}\in(\Z_+)^m$. We further deduce from  \eqref{111} that $\mathbf{s}^{\mathbf{p}}\mathbf{t}^{\mathbf{q}}\otimes v\in M$ for all $(\mathbf{p}, \mathbf{q})\in(\Z_+)^{2m}$ by induction on $\mathbf{p}$. That is $\C[\mathbf{s}, \mathbf{t}]\otimes v\subseteq M$. Let $V_0:=\{u\in V\mid \C[\mathbf{s}, \mathbf{t}]\otimes u\subseteq M\}$. The previous argument implies that $V_0\neq\{0\}$.  Moreover, $V_0$ is an $\mathfrak L$-submodule of  $V$. Since  $V$ is irreducible, it follows that $M=\bm{T}$, as desired.
\end{proof}
We are now in a position to determine a sufficient condition for the tensor product module $\bm{T}$ to be irreducible.	
\begin{theo}\label{theo1}
Let $\lambda_i\in\C^*$ for $i=1, 2, \ldots, m$,  with the $\lambda_i$ pairwise distinct. Then the tensor product module $\bm{T}$ is irreducible.
\end{theo}
\begin{proof}
Let $M$ be a nonzero submodule of $\bm{T}$, and choose a nonzero element $f(\mathbf{s}, \mathbf{t})=\sum_{(\mathbf{p}, \mathbf{q})\in E}\mathbf{s}^{\mathbf{p}}\mathbf{t}^{\mathbf{q}}\otimes v_{(\mathbf{p}, \mathbf{q})}\in M$ of minimal degree. We claim that ${\rm deg\,}(f(\mathbf{s}, \mathbf{t}))=\mathbf{0}$, consequently, $f(\mathbf{s}, \mathbf{t})=1\otimes\cdots \otimes 1\otimes v$ for some nonzero $v\in V$. Therefore, by Proposition \ref{prop3} we have $M=\bm{T}$, and hence $\bm{T}$ is irreducible.

Assume conversely that ${\rm deg\,}(f(\mathbf{s}, \mathbf{t}))\succ \mathbf{0}$.  We first show that $\mathbf{p}=\mathbf{0}$ for any $(\mathbf{p}, \mathbf{q})\in E$. If not, let $i_0:=\min\{i\mid p_i>0, (\mathbf{p}, \mathbf{q})\in E\ \text{for\ some}\ \mathbf{q}\}$. It follows from \eqref{444} that
$$0\neq\sum_{\substack{(\mathbf{p}, \mathbf{q})\in E \\ p_{i_0}=P_{i_0}}}\mathbf{s}^{\mathbf{p}-P_{i_0}\omega_{i_0}}\mathbf{t}^{\mathbf{q}}\otimes v_{(\mathbf{p}, \mathbf{q})}\in M,$$
which has lower degree than $f(\mathbf{s}, \mathbf{t})$, contradicting the minimality  of $f(\mathbf{s}, \mathbf{t})$. Hence, $\mathbf{p}=\mathbf{0}$ for any $(\mathbf{p}, \mathbf{q})\in E$. Combining this with the fact that $H_i(t_i)$ is a polynomial in $t_i$ and \eqref{222}, we get
\begin{equation}\label{vnm}\sum_{(\mathbf{0}, \mathbf{q})\in E}H_{i}(t_i)\mathbf{t}^{\mathbf{q}}\otimes v_{(\mathbf{0}, \mathbf{q})}\in M.\end{equation}
Now choose a minimal $j_0$ with $1\leq j_0\leq m$ such that $q_{j_0}>0$. Then \eqref{333} along with \eqref{vnm} forces $\sum_{(\mathbf{0}, \mathbf{q})\in E}q_{j_0}\mathbf{t}^{\mathbf{q}-\omega_{j_0}}\otimes v_{(\mathbf{0}, \mathbf{q})}\in M$, which is a nonzero element in $M$ and  has lower degree than $f(\mathbf{s}, \mathbf{t})$. This is a contradiction with the choice of $f(\mathbf{s}, \mathbf{t})$. Therefore, ${\rm deg\,}(f(\mathbf{s}, \mathbf{t}))=\mathbf{0}$, as desired.
\end{proof}
We now turn to the case where $\lambda_1, \ldots, \lambda_m$ are not pairwise distinct. Actually,  $\bm{T}$ is reducible in this case and it suffices to show the reducibility of $\Phi(\lambda,\a_1,\b_1,\r_1,h_1(t_1))\otimes\Phi(\lambda,\a_2,\b_2,\r_2,h_2(t_2))$. We still use the notations as before, only taking $m=2$ and $V$ as the $1$-dimensional trivial module. For convenience, we identify $\Phi(\lambda,\a_1,\b_1,\r_1,h_1(t_1))\otimes\Phi(\lambda,\a_2,\b_2,\r_2,h_2(t_2))=\C[s_1, s_2, t_1, t_2]$. For any $l\in\Z_+$, denote$$N_l={\rm span\,}\{s_1^r(s_1+s_2)^u\C[t_1, t_2]\mid r, u\in\Z_+, r\leq l\}.$$
Clearly, $N_l\subset N_{l+1}$. Moreover, we have the following.
\begin{prop}\label{theo2}
Keep notations as above, then each $N_l$ is a proper submodule of  $$\Phi(\lambda,\a_1,\b_1,\r_1,h_1(t_1))\otimes\Phi(\lambda,\a_2,\b_2,\r_2,h_2(t_2)).$$
\end{prop}
\begin{proof}
Fix any $l\in\Z_+$. For any element of the form $s_1^r(s_1+s_2)^ug_1(t_1)g_2(t_2)\in N_l$, with $r, u\in\Z_+, r\leq l, g_1(t_1)\in\C[t_1], g_2(t_2)\in\C[t_2]$, we can compute
\begin{align*}
&\lambda^{-n} L_n\big( s_1^r (s_1 + s_2)^u g_1(t_1) g_2(t_2) \big)\\
&=\lambda^{-n} L_n \Big( \sum_{i=0}^u \binom{u}{i} s_1^{r+u-i} s_2^i g_1(t_1) g_2(t_2) \Big)\\
&= \lambda^{-n} \sum_{i=0}^u \binom{u}{i} \Big( L_n(s_1^{r+u-i} g_1(t_1)) s_2^i g_2(t_2)+ s_1^{r+u-i} g_1(t_1) L_n(s_2^i g_2(t_2)) \Big)\\
&=\sum_{i=0}^u \binom{u}{i}s_2^i g_2(t_2)\big((s_1+h_{(n)}(t_1))(s_1-n)^{r+u-i}g_1(t_1)-n(t_1-n\a_1)(s_1-n)^{r+u-i}g_1'(t_1)\big)\\
&\quad +\sum_{i=0}^u \binom{u}{i}s_1^{r+u-i}g_1(t_1)\big((s_2+h_{(n)}(t_2))(s_2-n)^{i}g_2(t_2)-n(t_2-n\a_2)(s_2-n)^{i}g_2'(t_2)\big)\\
&= (s_1(s_1 - n)^r-s_1^{r+1}) (s_1 + s_2 - n)^u g_1(t_1) g_2(t_2) + s_1^r (s_1+s_2) (s_1 + s_2 - n)^u g_1(t_1) g_2(t_2)\\
&\quad + (s_1 - n)^r (s_1 + s_2 - n)^u h_{(n)}(t_1) g_1(t_1) g_2(t_2) + s_1^r (s_1 + s_2 - n)^u h_{(n)}(t_2) g_1(t_1) g_2(t_2)\\
&\quad - (s_1 - n)^r (s_1 + s_2 - n)^u n (t_1 - n \a_1) g_1'(t_1) g_2(t_2) - s_1^r (s_1 + s_2 - n)^u n(t_2 -n\a_2) g_1(t_1) g_2'(t_2),
\end{align*}
\begin{align*}
&\lambda^{-n} W_n\big( s_1^r (s_1 + s_2)^u g_1(t_1) g_2(t_2) \big) &&\\
%&= \lambda^{-n} W_n \Big( \sum_{i=0}^u \binom{u}{i} s_1^{r+u-i} s_2^i g_1(t_1) g_2(t_2) \Big) &&\\
&= \lambda^{-n} \sum_{i=0}^u \binom{u}{i} \Big( W_n(s_1^{r+u-i} g_1(t_1)) s_2^i g_2(t_2) + s_1^{r+u-i} g_1(t_1) W_n(s_2^i g_2(t_2)) \Big) &&\\
&=\sum_{i=0}^u \binom{u}{i}\big((t_1-n\a_1)(s_1-n)^{r+u-i}g_1(t_1)s_2^i g_2(t_2)+s_1^{r+u-i}g_1(t_1)(t_2-n\a_2)(s_2-n)^i g_2(t_2)\big)&&\\
&= (s_1 - n)^r (s_1 + s_2 - n)^u (t_1 - n a_1) g_1(t_1) g_2(t_2) + s_1^r (s_1 + s_2 - n)^u (t_2 - na_2) g_1(t_1) g_2(t_2),
\end{align*}
\begin{align*}
&\lambda^{-n} I_n\big( s_1^r (s_1 + s_2)^u g_1(t_1) g_2(t_2) \big)\\
%&=\lambda^{-n} L_n \Big( \sum_{i=0}^u \binom{u}{i} s_1^{r+u-i} s_2^i g_1(t_1) g_2(t_2) \Big)\\
&= \lambda^{-n} \sum_{i=0}^u \binom{u}{i} \Big( I_n(s_1^{r+u-i} g_1(t_1)) s_2^i g_2(t_2)+ s_1^{r+u-i} g_1(t_1) I_n(s_2^i g_2(t_2)) \Big)\\
&=\sum_{i=0}^u \binom{u}{i}s_2^i g_2(t_2)(s_1-n)^{r+u-i}\big(\r_1g_1(t_1)+n\b_1(g_1'(t_1)-g_1(t_1)H_1(t_1))\big)\\
&\quad +\sum_{i=0}^u \binom{u}{i}s_1^{r+u-i}g_1(t_1)(s_2-n)^{i}\big(\r_2g_2(t_2)+n\b_2(g_2'(t_2)-g_2(t_2)H_2(t_2))\big)\\
&=(s_1 - n)^r(s_1 + s_2 - n)^u g_2(t_2)\big(\r_1g_1(t_1)+n\b_1(g_1'(t_1)-g_1(t_1)H_1(t_1))\big)\\
&\quad + s_1^r(s_1 + s_2 - n)^ug_1(t_1)\big(\r_2g_2(t_2)+n\b_2(g_2'(t_2)-g_2(t_2)H_2(t_2))\big) 
\end{align*}
and
\begin{align*}
&\lambda^{-n} J_n\big( s_1^r (s_1 + s_2)^u g_1(t_1) g_2(t_2) \big)\\
%&= \lambda^{-n} J_n \Big( \sum_{i=0}^u \binom{u}{i} s_1^{r+u-i} s_2^i g_1(t_1) g_2(t_2) \Big) &&\\
&= \lambda^{-n} \sum_{i=0}^u \binom{u}{i} \Big( J_n(s_1^{r+u-i} g_1(t_1)) s_2^i g_2(t_2) + s_1^{r+u-i} g_1(t_1) J_n(s_2^i g_2(t_2)) \Big)\\
&=\sum_{i=0}^u \binom{u}{i}\big(\b_1(s_1-n)^{r+u-i}g_1(t_1) s_2^i g_2(t_2)+\b_2s_1^{r+u-i} g_1(t_1)(s_2-n)^i g_2(t_2)\big)\\
&= \b_1(s_1 - n)^r (s_1 + s_2 - n)^u  g_1(t_1) g_2(t_2) + \b_2s_1^r (s_1 + s_2 - n)^u g_1(t_1) g_2(t_2).
\end{align*}
These imply that $$X_n\big(s_1^r(s_1+s_2)^ug_1(t_1)g_2(t_2)\big)\in N_l, \quad{\rm where}\,\,X_n\in\{L_n, W_n,I_n,J_n\mid\forall \,n\in\Z\}.$$
Hence, $N_l$ is a submodule of $\Phi(\lambda,\a_1,\b_1,\r_1,h_1(t_1))\otimes\Phi(\lambda,\a_2,\b_2,\r_2,h_2(t_2))$, completing the proof.
	\end{proof}
As a direct consequence of Theorem \ref{theo1} and Proposition \ref{theo2}, we have the following sufficient and necessary condition for a tensor product module to be irreducible.
\begin{coro}\label{theo3}
The $\mathfrak L$-module  $\bm{T}$ is irreducible if and only if $\lambda_1, \ldots, \lambda_m$ are pairwise distinct.
\end{coro}	
\section{Isomorphism classes of the tensor product modules}
In this section, we will determine the necessary and sufficient conditions for two irreducible tensor product modules $\bm{T}=(\otimes_{i=1}^{m}\Phi(\lambda_i, \alpha_i, \b_i, \r_i, h_i(t_i))\otimes V$ to be isomorphic. By Corollary \ref{theo3}, to ensure that $\bm{T}$ is irreducible, we suppose that $\lambda_1, \ldots, \lambda_m\in\C^*$ are pairwise distinct throughout this section. 

For any $f:=f(\mathbf{s}, \mathbf{t})\in \bm{T}$, we define
\begin{equation*} 
R_f = \lim\limits_{l\to\infty} \operatorname{rank} \{f, W_n(f), J_n(f)\mid n\ge l\}
\end{equation*}
and
$$R_{\bm{T}} = {\rm inf\,}\{R_{f}\mid 0\neq f\in \bm{T}\},$$
where ${\rm rank\,}(Y)={\rm dim\,}{\rm span\,}(Y)$ for any subset $Y$ in a vector space. If we write  $f$ in the form \eqref{a12}, there exists a minimal positive integer $K(f)$ such that $W_nv_{(\mathbf{p}, \mathbf{q})}=J_nv_{(\mathbf{p}, \mathbf{q})}=0$ for all $n\geq K(f)$ and $(\mathbf{p}, \mathbf{q})\in E$.

Let $\bm{Q}:=\otimes_{i=1}^m\C[t_i]\otimes V\subset \bm{T}$. We have the following result describing the property of the invariants $R_{f}$ and $R_{\bm{T}}$.
\begin{lemm}\label{theoo5}
For any nonzero $f\in \bm{T}$, the following statements hold.
\begin{trivlist}
\item[(1)]For all $l\geq K(f)$, $R_f=\operatorname{rank} \{f,  W_n(f), J_n(f)\mid n\ge l\}$.
\item[(2)]$R_f\geq m+1$ and the equality holds if and only if  $0\neq f\in \bm{Q}$.
\item[(3)]$R_{\bm{T}}=m+1$.
\end{trivlist}
\end{lemm}
\begin{proof}
Denote $R_{f, l}= {\rm rank\,}\{f, W_n(f), J_n(f)\mid n\geq l\}$ for any $l\in\N$, then it suffices to show that $R_{f, l}=R_{f, K(f)}$ for all $l\geq K(f)$. For $1\leq i \leq m, 0\leq z \leq P_i+1, 0\leq k \leq P_i$, denote
\begin{align*}a_{i,z}&:=\sum_{(\mathbf{p}, \mathbf{q})\in E}s_1^{p_1} t_1^{q_1}\otimes\cdots\otimes\big(\binom{p_i}{z} s_i^{p_{i}-z} t_i^{q_i+1}+ \binom{p_i}{z-1} \a_i s_i^{p_i -z +1}t_i^{q_i}\big)\otimes\cdots\otimes s_m^{p_m}t_m^{q_m}\otimes v_{(\mathbf{p}, \mathbf{q})},\\
b_{i,k}&:=\sum_{(\mathbf{p}, \mathbf{q})\in E}s_1^{p_1} t_1^{q_1}\otimes\cdots\otimes\binom{p_i}{k}\b_i s_i^{p_{i}-k} t_i^{q_i}\otimes\cdots\otimes s_m^{p_m}t_m^{q_m}\otimes v_{(\mathbf{p}, \mathbf{q})}.
\end{align*}
For any $n\geq l\geq K(f)$, it follows from \eqref{bbn2} and
\begin{align*}
W_n(f)&=\sum_{(\mathbf{p}, \mathbf{q})\in E}\sum_{i=1}^ms_1^{p_1}t_1^{q_1}\otimes\cdots \otimes W_n(s_i^{p_{i}}t_i^{q_i})\otimes\cdots\otimes s_m^{p_m}t_m^{q_m}\otimes v_{(\mathbf{p}, \mathbf{q})}\nonumber\\
&=\sum_{(\mathbf{p}, \mathbf{q})\in E}\sum_{i=1}^ms_1^{p_1}t_1^{q_1}\otimes\cdots \otimes \big(\lambda_i^n(t_i-n\a_i)(s_i-n)^{p_i} t_i^{q_i}\big)\otimes\cdots\otimes s_m^{p_m}t_m^{q_m}\otimes v_{(\mathbf{p}, \mathbf{q})}\nonumber\\
&=\sum_{(\mathbf{p}, \mathbf{q})\in E}\sum_{i=1}^{m}\sum_{z=0}^{p_i+1}(-1)^z n^z\lambda_i^n s_1^{p_1} t_1^{q_1}\otimes\cdots\otimes\big(\binom{p_i}{z} s_i^{p_{i}-z} t_i^{q_i+1} + \binom{p_i}{z-1} \a_i s_i^{p_i -z+1}t_i^{q_i}\big)\otimes\cdots\nonumber\\
&\quad\otimes s_m^{p_m}t_m^{q_m}\otimes v_{(\mathbf{p}, \mathbf{q})}
\end{align*}
that 
$${\rm span\,}\{f, W_n(f), J_n(f)\mid n\geq l\}={\rm span\,}\{f, a_{i, z}, b_{i, k}\mid 0\leq z\leq  P_i+1, 0\leq k \leq P_i, 1\leq i \leq m\},$$
proving (1).

(2) If $f\in\bm{Q}$, then we have $R_{f}={\rm rank\,}\{f, a_{i, 0}\mid 1\leq i\leq m\}=m+1$. Now suppose ${\rm deg\,}(f)=(\mathbf{p}^\prime, \mathbf{q}^\prime)=(p_1^\prime,\ldots, p_m^\prime,  q_1^\prime,\ldots, q_m^\prime)$ with $p_i^\prime>0$ for some $1\leq i\leq m$. Let $i_1={\rm min}\{i\mid p_i^\prime>0\}$.  Then we have ${\rm deg\,}(a_{i_1, P_{i_1}+1})\prec (\mathbf{p}^\prime, \mathbf{q}^\prime)$ if $\a\in\C^*$, or ${\rm deg\,}(b_{i_1, P_{i_1}})\prec (\mathbf{p}^\prime, \mathbf{q}^\prime)$ if $\b\in\C^*$.
Hence the space spanned by $f, a_{i, 0}, a_{i_1, P_{i_1}+1}, b_{i_1, P_{i_1}}, 1\leq i\leq m$ has dimension $m+2$, which means $R_{f}\geq m+2$. This proves (2), from which (3) follows trivially.
\end{proof}
	
Now we are ready to prove our isomorphism theorem. Let $\bm{T}$ be the tensor module defined before. Now we take another tensor module,
\begin{equation*}
\bm{T'}:=\bm{T'}(\bm{\lambda}', \bm{\alpha}',  \bm{\b}', \bm{\r}', \bm{g(t)}, V')=(\otimes_{j=1}^{m'}\Phi(\lambda_j', \alpha_j', \b_j', \r_j', g_j(t_j)))\otimes V',
\end{equation*}
where $\bm{\lambda}'=(\lambda_1', \ldots, \lambda_{m'}')$, $\bm{\alpha}'=(\alpha_1', \ldots, \alpha_{m'}')$, $\bm{\b}'=(\b_1', \ldots, \b_{m'}')$, $\bm{\r}'=(\r_1', \ldots, \r_{m'}')$, $\bm{g(t)}=(g_1(t_1), \ldots, g_{m'}(t_{m'}))$ and $V'$ is an irreducible restricted module. We identify $$\Phi(\lambda_j', \alpha_j', \b_j', \r_j', g_j(t_j)))=\C[s_j, t_j]$$ for $1\leq j\leq m'$. Let $G_j(t_j)=\frac{g_j(t_j)-g_j(\a_j')}{t_j-\a_j'}$. To ensure that $\bm{T}$ and $\bm{T'}$ are irreducible, we suppose that $\lambda_1, \ldots, \lambda_{m}$  are pairwise distinct 
as well as  $\lambda_1', \ldots, \lambda_{m'}'$ are pairwise distinct.
	
\begin{theo}\label{theo4}
$\bm{T}\cong \bm{T'}$ as $\mathfrak L$-modules if and only if $m=m', V\cong V'$ and 
$$(\lambda_i, \alpha_i, \b_i, \r_i,  h_i(t_i))=(\lambda_i', \alpha_i', \beta_i', \gamma_i',g_i(t_i))$$ for $1\leq i\leq m$ by renumbering the indices $(\lambda_i', \alpha_i', \beta_i', \gamma_i', g_i(t_i))$ if necessary.
\end{theo}	
\begin{proof}
The sufficiency is obvious and it suffices to show the necessity. Let $\phi$ be an
$\mathfrak L$-module isomorphism from $\bm{T}$ to $\bm{T'}$. From Lemma \ref{theoo5} (3), we see that $m+1=R_{\bm{T}}=R_{\bm{T'}}=m'+1$, i.e., $m=m'$. This along with Lemma \ref{theoo5}(2) entails us to assume that
$$\phi(1\otimes\cdots \otimes1 \otimes v)=\sum_{(\mathbf{0}, \mathbf{q})\in E}t_1^{q_1}\otimes\cdots \otimes t_m^{q_m} \otimes v_{(\mathbf{0}, \mathbf{q})}$$
for a fixed $0\neq v\in V$. Using the same argument as in   \cite[Theorem 3.6]{CX}, we obtain that $\lambda_i=\lambda_i', \a_i=\a_i'$ for $1\leq i \leq m$ and 
\begin{equation}\label{bmm}
\phi(t_1^{r_1}\otimes\cdots \otimes t_m^{r_m} \otimes v)=\sum_{(\mathbf{0}, \mathbf{q})\in E}t_1^{r_1+q_1}\otimes\cdots \otimes t_m^{r_m+q_m} \otimes v_{(\mathbf{0}, \mathbf{q})}
\end{equation}
for any $r_1, r_2, \ldots, r_m\in\Z_+$.

Choose  $k$ sufficiently large such that $\mathfrak L_n(v)=\mathfrak L_n(v_{(\mathbf{0}, \mathbf{q})})=0$ for any $n\geq k$ and $(\mathbf{0}, \mathbf{q})\in E$. For any $n\geq k$,  using \eqref{bmm}, explicit calculations show that    
\begin{align}\label{vvc1}
0&=\phi\big(J_n(1\otimes\cdots \otimes1 \otimes v)\big)-J_n\big(\phi(1\otimes\cdots \otimes1 \otimes v)\big)\nonumber\\
&=\sum_{i=1}^m\lambda_i^n (\b_{i}-\b_{i}')\phi(1\otimes\cdots \otimes  1\otimes v),
\end{align}
\begin{eqnarray}\label{vvc2}
0&=&\phi\big(I_n(1\otimes\cdots \otimes1 \otimes v)\big)-I_n\big(\phi(1\otimes\cdots \otimes1 \otimes v)\big)\nonumber\\
&=&\sum_{i=1}^m\sum_{(\mathbf{0}, \mathbf{q})\in E}t_1^{q_1}\otimes\cdots\otimes
\lambda_i^n(\r_i-\r_{i}')t_i^{q_i}\otimes\cdots \otimes t_{m}^{q_m} \otimes v_{(\mathbf{0}, \mathbf{q})}\nonumber\\
&&+\sum_{i=1}^m\sum_{(\mathbf{0}, \mathbf{q})\in E}t_1^{q_1}\otimes\cdots\otimes
\lambda_i^n n\big((\b_{i}'G_{i}(t_{i})-\b_{i}H_{i}(t_{i}))t_i^{q_i}-\b_{i}'q_{i}t_{i}^{q_i-1}\big)\otimes\cdots \otimes t_{m}^{q_m} \otimes v_{(\mathbf{0}, \mathbf{q})}
\end{eqnarray}
and that
\begin{eqnarray}\label{vvc3}
0 &=&\phi\bigl(L_n(1\otimes\cdots \otimes 1\otimes v)\bigr)-L_n\big(\phi(1\otimes\cdots \otimes 1\otimes v)\big) \nonumber \\
&=&\sum_{i=1}^m \lambda_i^n \big(\phi(1\otimes\cdots \otimes s_i\otimes\cdots \otimes 1\otimes v)-\sum_{(\mathbf{0}, \mathbf{q})\in E}t_1^{q_1}\otimes\cdots\otimes s_it_{i}^{q_i}
\otimes\cdots \otimes t_{m}^{q_m} \otimes v_{(\mathbf{0}, \mathbf{q})}\big)\nonumber\\
&&+\sum_{i=1}^m\sum_{(\mathbf{0}, \mathbf{q})\in E}t_1^{q_1}\otimes\cdots\otimes
\lambda_i^n n(t_iH_i(t_i)-t_iG_i(t_i)+h_i(\a_i)-g_i(\a_i)+q_i)t_i^{q_i}\otimes\cdots\nonumber\\
&&\otimes t_{m}^{q_m} \otimes v_{(\mathbf{0}, \mathbf{q})}+\sum_{i=1}^m\sum_{(\mathbf{0}, \mathbf{q})\in E}t_1^{q_1}\otimes\cdots\otimes\lambda_i^n n^2\a_i(G_i(t_i)t_i^{q_i}-H_i(t_i)t_i^{q_i}-q_it_i^{q_i-1})\otimes\cdots\nonumber\\
&&\otimes t_{m}^{q_m} \otimes v_{(\mathbf{0}, \mathbf{q})}.
\end{eqnarray}
From Lemma \ref{prop1}, we know that for any $1\leq i \leq m$, the coefficients of $\lambda_i^n$ in \eqref{vvc1} and  \eqref{vvc2} should both be zero, forcing  $\b_i=\b_{i}'$ and $\r_i=\r_{i}'$, respectively.  For the case $\b_i\in\C^*$,  observing that  the coefficients of  $n\lambda_i^n$  in \eqref{vvc2} and  \eqref{vvc3} both vanish gives
\begin{equation*}(G_{i}(t_{i})-H_{i}(t_{i}))t_i^{q_i}-q_{i}t_{i}^{q_i-1}=0 \quad {\rm and}\quad t_iH_i(t_i)-t_iG_i(t_i)+h_i(\a_i)-g_i(\a_i)+q_i=0,\end{equation*}
from which we infer  that 
\begin{equation}\label{mm0}h_i(t_i)=g_i(t_i)\quad {\rm and} \quad q_i=0,\ \ \ \ \forall\,1\leq i \leq m. \end{equation}
For the remaining case  $\a_i\in\C^*$, \eqref{mm0} were obtained in \cite[Theorem 3.6]{CX} by considering, respectively, the coefficients at $n\lambda_i^n$  and $n^2\lambda_i^n$ in \eqref{vvc3}. In both cases, we have
\begin{equation}\label{mm2}\phi(1\otimes\cdots \otimes1 \otimes v)=1\otimes\cdots \otimes1 \otimes \tau(v)\end{equation}
for some $\tau(v)\in V'$. Also, since the coefficient of $\lambda_i^n$ in  \eqref{vvc3}  vanishes, it follows that
\begin{equation*}\label{mm3}\phi(1\otimes\cdots \otimes s_i\otimes\cdots \otimes1 \otimes v)=1\otimes\cdots \otimes s_i\otimes\cdots1 \otimes \tau(v), \ \ \ \ \forall\,1\leq i\leq m.\end{equation*}
Combining this with \eqref{bmm} and \eqref{mm2} yields 
\begin{equation*}
\phi(X_n (1\otimes\cdots \otimes1)\otimes v)=X_n(1\otimes\cdots \otimes1 )\otimes\tau(v), \quad{\rm where}\,\,X_n\in\{L_n, W_n, I_n, J_n\mid\forall\,n\in\Z\}.
\end{equation*}
This together with
$$\phi(X_n (1\otimes\cdots \otimes1\otimes v))=X_n (\phi(1\otimes\cdots \otimes1\otimes v)),\quad{\rm where}\,\,X_n\in\{L_n, W_n, I_n, J_n\mid\forall\,n\in\Z\}$$
gives
$$\phi(1\otimes\cdots \otimes1\otimes X_n( v))=1\otimes\cdots \otimes1\otimes X_n(\tau(v)),\quad{\rm where}\,\,X_n\in\{L_n, W_n, I_n, J_n\mid\forall\,n\in\Z\}.$$
Therefore,
$$\tau(X_n (v))=X_n(\tau(v)),\quad{\rm where}\,\,X_n\in\{L_n, W_n, I_n, J_n\mid\forall\, n\in\Z\}, v\in V.$$
Thus, $\tau$ is a nonzero $\mathfrak L$-module homomorphism. Since $V$ and $V'$ are irreducible $\mathfrak L$-modules, $\tau$ is an  $\mathfrak L$-module isomorphism.  We complete the proof.
\end{proof}
	
\section{Comparison of tensor product modules with known non-weight modules}
In this section, we compare the tensor product modules constructed in the present paper with all other known non-weight $\mathfrak L$-modules, i.e., restricted modules and $\Phi(\lambda,\a,\b,\r,h(t))$. We fix an irreducible tensor module $\bm{T}$ as defined in \eqref{a11}, where $\lambda_i\in \C^*$ for $i=1, 2, \ldots, m$ with the $\lambda_i$ pairwise distinct.
	
For any $r\in\Z_+, l, k\in\Z$, as in \cite{LLZ}, we denote
$$\omega_{l,k}^{(r)}=\sum_{i=0}^r\binom{r}{i}(-1)^{r-i}L_{l-k-i}L_{k+i}\in U(\mathfrak L).$$
\begin{lemm}\label{theo123456} Keep notations as above. Then the following statements hold.
\begin{trivlist}
\item[(1)] For sufficiently large $n\in\N$, the action of $L_n$ on  $\bm{T}$ is not locally finite.
\item[(2)] For any $f(s,t) \in \Phi(\lambda, \alpha, \b, \r, h(t))$, we have $\omega_{l,k}^{(r)}(f(s,t)) = 0, \forall\,\, l,k\in\Z, r > 4$.
\item[(3)] If $V$ is not the $1$-dimensional trivial module, then for any  $r>4$, there exist $l, k\in\Z$ and $v\in V$ such that $\omega_{l,k}^{(r)}(1\otimes\cdots\otimes1\otimes v)\neq0$.
\item[(4)] Assume that $m\geq2$ and $V$ is the $1$-dimensional trivial module. Then for any $r>4$, there exist $l, k\in\Z$ such that $\omega_{l,k}^{(r)}(1\otimes\cdots\otimes1)\neq0$.
\end{trivlist}
\end{lemm}
\begin{proof}
%For any $f(\mathbf{s}, \mathbf{t})\in \bm{T}$ and sufficiently large $n\in\N$, it follows from Proposition \ref{prop2} that ${\rm deg\,}(L_n(f(\mathbf{s}, \mathbf{t})))\succ {\rm deg\,}(f(\mathbf{s}, \mathbf{t}))$, so we have $f(\mathbf{s}, \mathbf{t}), L_n(f(\mathbf{s}, \mathbf{t})), L_n^2(f(\mathbf{s}, \mathbf{t})), \ldots$ are linearly independent. Hence, (1) follows.
Since the proof is similar to those of \cite[Lemma 5.1]{LGW} and \cite[Lemma 4.1]{GLW}, we omit the details.
\end{proof}
\begin{theo}\label{theovvn}
The tensor product module $\bm{T}$ is a new non-weight $\mathfrak L$-module except when $m=1$ and $V$ is the $1$-dimensional trivial module.
\end{theo}
\begin{proof}
Given an irreducible restricted module $S$, the action of $L_n$ on $S$ is locally finite for sufficiently large $n\in\N$ by \cite[Lemma 4.1]{LS}. Together with  Lemma \ref{theo123456}(1), we have $\bm{T}\not\cong S$. From Lemma \ref{theo123456}(2)-(4), we see  that $\bm{T}\not\cong \Phi(\lambda,\a,\b,\r,h(t))$. This completes the proof.
\end{proof}

\end{document}